\documentclass[11pt]{article}

\usepackage[margin=1in]{geometry}
\usepackage{amsthm,amsmath}
\usepackage{latexsym,amssymb}
\usepackage{graphicx,float,color,fancybox,shapepar,setspace}
\usepackage{subfigure}
\usepackage{pgf,tikz}
\usepackage[affil-it]{authblk}
\usepackage{indentfirst}
\usepackage[section]{placeins}
\usepackage{enumitem}
\usepackage{lineno}
\usepackage{hyperref}

\hypersetup{
	colorlinks=true,
	linkcolor=blue,
	filecolor=blue,
	urlcolor=blue,
	citecolor=cyan,
}

\usetikzlibrary{arrows}
\makeatletter
\renewenvironment{proof}[1][\proofname]{\par
	\normalfont \topsep6\p@\@plus6\p@\relax
	\trivlist
	\item[\hskip\labelsep
	\itshape
	#1\@addpunct{.}]\ignorespaces
}{%
	\hfill $\square$
	\endtrivlist\@endpefalse
}
\makeatother

\newtheorem{theorem}{Theorem}[section]

\newtheorem{claim}{Claim}[section]

\newtheorem{conjecture}[theorem]{Conjecture}

\begin{document}
	
	\title{Counterexamples to the Balogh-Linz-Patk\'os Conjecture}
	
	\author[1]{Jia-Bao Yang\footnote{Email: \texttt{jbyang1215@nju.edu.cn}}}
	
	\author[2,3]{Leilei Zhang\footnote{Corresponding author. Email: \texttt{mathdzhang@163.com}}}
	
	\date{}
	\affil[1]{\small School of Mathematics, Nanjing University, Nanjing, 210093,  P.R. China }
	\affil[2]{\footnotesize School of Mathematics and Statistics, Central China Normal University, Wuhan 430079, China}
	\affil[3]{\small Faculty of Environment and Information Sciences, Yokohama National University, Yokohama 240-8501, Japan}
	\maketitle

	\begin{abstract}
		A set system $\mathcal{F}$ is called $t$-{\it intersecting} if $|A\cap B|\ge t$ for every pair of sets $A,B\in \mathcal{F}$.  A set system $\mathcal{F}$ is $k$-{\it Sperner} if it does not contain a chain of length $k+1$. Balogh, Linz and Patk\'os (Combinatorial Theory, 2023) conjectured an extremal result for $t$-intersecting $k$-Sperner families when $n+t$ is odd.  In this note we give an explicit construction that is $t$-intersecting and $k$-Sperner, and whose size exceeds that of the conjectured fixed-star construction for infinitely many values of $n$. Consequently, we disprove the Balogh-Linz-Patk\'os conjecture for all $t\ge 2$ and $k\ge 2$ satisfying $k(t-1)\ge t+1$.
	\end{abstract}
	

	\section{Introduction}
	
	For a positive integer $n$, we write $[n]=\{1,2,\dots,n\}$ and $2^{[n]}$ for the power set of $[n]$.
	For a set $V$ and an integer $r$, let $\binom{V}{r}$ denote the family of all $r$-subsets of $V$.  
	A family of sets $\mathcal{F}\subseteq 2^{[n]}$ is \emph{$t$-intersecting} if $|A\cap B|\ge t$ for all $A,B\in \mathcal{F}$. 
	In the special case $t=1$, such a family is simply called intersecting.  
	The family $\mathcal{F}$ is \emph{$k$-Sperner} if it contains no chain of length $k+1$; 
	equivalently, there do not exist sets $A_0\subset A_1\subset\cdots\subset A_k$ with all $A_i\in\mathcal{F}$.  
	Thus a $1$-Sperner family is precisely a Sperner family, or antichain.
	
	The classical theorem of Sperner \cite{Sperner1928} determines the maximum possible size of an antichain in $2^{[n]}$ as follows.
	
	\begin{theorem}[Sperner \cite{Sperner1928}] 
		Let $\mathcal{F}\subseteq 2^{[n]}$ be an antichain.  Then
		$$|\mathcal{F}|\le \binom{n}{\lfloor n/2\rfloor}.$$
		Moreover, equality holds only if $\mathcal{F}$ is one of the largest layers in the Boolean lattice $2^{[n]}$.
	\end{theorem}
	
	Later, Sperner's theorem was extended to families with no long chains by Erd\H{o}s \cite{Erdos1945}.
	Another natural extension of Sperner's theorem, due to Milner \cite{Milner1968}, imposes an intersection condition in addition to the antichain condition.
	
	\begin{theorem}[Milner \cite{Milner1968}] 
		If $\mathcal{F}\subseteq 2^{[n]}$ is a $t$-intersecting antichain, then
		$$|\mathcal{F}|\le \binom{n}{\lfloor (n+t+1)/2\rfloor}.$$
	\end{theorem}
	
	Along a different line of research, Frankl \cite{Frankl1990} determined the maximum size of an intersecting $k$-Sperner family.  Different proofs were later given by Gerbner \cite{Gerbner2013} and by Gerbner, Methuku, and Tompkins \cite{Gerbner2017}.
	
	\begin{theorem}[Frankl \cite{Frankl1990}]
		Let $\mathcal{F}\subseteq 2^{[n]}$ be an intersecting $k$-Sperner family.  Then
		$$
		|\mathcal{F}| \le
		\begin{cases}
			\displaystyle \sum_{i=(n+1)/2}^{(n+1)/2+k-1}\binom{n}{i}, & \text{if $n$ is odd,} \\
			\displaystyle \binom{n-1}{n/2-1}+\sum_{i=n/2+1}^{n/2+k-1}\binom{n}{i}+\binom{n-1}{n/2+k}, & \text{if $n$ is even.}
		\end{cases}
		$$
		Furthermore, the equality case was characterized.
	\end{theorem}
	
	A natural common generalization of Milner's theorem and Frankl's theorem asks for the largest possible size of a $t$-intersecting and $k$-Sperner family. 
	Frankl \cite{Frankl2021} formulated conjectures for this problem and proved them in several ranges.  
	The predicted extremal construction depends on the parity of $n+t$. 
	If $n+t$ is even, the conjectured answer has a particularly simple form: one takes the $k$ consecutive layers beginning at level $(n+t)/2$.
	
	\begin{conjecture}[Frankl \cite{Frankl2021}]\label{conj}
		Suppose that $n+t$ is even and $n>t$.  If $\mathcal{F}\subseteq 2^{[n]}$ is a $t$-intersecting $k$-Sperner family, then
		$$
		|\mathcal{F}|\le \sum_{i=0}^{k-1}\binom{n}{(n+t)/2+i}.
		$$
	\end{conjecture}
	
	An easy observation shows that Conjecture \ref{conj} is clearly sharp, as shown by the family
	$$\bigcup_{i=0}^{k-1}\binom{[n]}{(n+t)/2+i}.$$
	
	Frankl proved Conjecture \ref{conj} when $t\ge n-O(\sqrt n)$.  
	Balogh, Linz, and Patk\'os \cite{Balogh2023} recently proved Conjecture \ref{conj} for every fixed $t$ and all sufficiently large $n$.
	
	\begin{theorem}[Balogh, Linz and Patk\'os \cite{Balogh2023}]\label{thm:Balogh2023}
		Let $t$ and $k$ be positive integers.  Suppose that $n+t$ is even, $t\le n$, and $n$ is sufficiently large.  If $\mathcal{F}\subseteq 2^{[n]}$ is a $t$-intersecting $k$-Sperner family, then
		$$|\mathcal{F}|\le \sum_{i=0}^{k-1}\binom{n}{(n+t)/2+i}.$$
	\end{theorem}
	
	The expected extremal structure is more delicate when $n+t$ is odd.  
	If $n+t$ is odd, throughout this paper, we let
	$$q=\frac{n+t-1}{2}.$$
	Fix a $t$-set $T\subseteq[n]$, define the $q$-uniform $t$-star as follows:
	$$\mathcal{S}_T=\left\{S\in\binom{[n]}{q}:T\subseteq S\right\}.$$
	For a family $\mathcal{F}\subseteq\binom{[n]}{r}$ and an integer $s\ge r$, we denote the $s$-th upper shadow of $\mathcal{F}$ by
	$$
	\nabla_s(\mathcal{F})=\left\{A\in\binom{[n]}{s}:\text{ there exists }F\in\mathcal{F}\text{ with }F\subseteq A\right\}.
	$$
	The fixed-star candidate can then be written as
	\begin{equation}\label{eq1}
		\mathcal{B}_T(t,k)
		=\mathcal{S}_T
		\cup\bigcup_{i=1}^{k-1}\binom{[n]}{q+i}
		\cup\left(\binom{[n]}{q+k}\setminus \nabla_{q+k}(\mathcal{S}_T)\right).
	\end{equation}
	Although the family itself depends on $T$, its size does not.  Indeed,
	\begin{equation}\label{eq2}
		|\mathcal{B}_T(t,k)|
		=\binom{n-t}{q-t}
		+\sum_{i=1}^{k}\binom{n}{q+i}
		-\binom{n-t}{q+k-t}.
	\end{equation}
	
	Balogh, Linz and Patk\'os \cite{Balogh2023} conjectured that, for fixed $t$ and $k$, and for all sufficiently large $n$ with $n+t$ odd, every $t$-intersecting $k$-Sperner family has size at most $|\mathcal{B}_T(t,k)|$.
	
	\begin{conjecture}[Balogh, Linz and Patk\'os \cite{Balogh2023}]\label{conj:oddcase}
		There exists a positive integer $n_0=n_0(k,t)$ such that if $n+t$ is odd, $n>n_0$, and $\mathcal{F}\subseteq 2^{[n]}$ is a $t$-intersecting $k$-Sperner family, then
		$$
		|\mathcal{F}|\le\binom{n-t}{q-t}+\sum_{i=1}^{k}\binom{n}{q+i}-\binom{n-t}{q+k-t},
		$$
		where $q=(n+t-1)/2$.
	\end{conjecture}
	
	Now we are ready to state our main result.
	
	\begin{theorem}\label{thm1}
		Let $t\ge 2$ and $k\ge 2$ be integers satisfying $k(t-1)\ge t+1$.  
		For every integer $m\ge k$, let $n=t+2m+1$ and $q=(n+t-1)/2$.
		Then, for every fixed $t$-set $T\subseteq[n]$, there exists a $t$-intersecting $k$-Sperner family $\mathcal{F}\subseteq 2^{[n]}$ such that $|\mathcal{F}|>|\mathcal{B}_T(t,k)|$.
	\end{theorem}
	
	For each fixed pair $(t,k)$ satisfying the hypotheses of Theorem \ref{thm1}, the parameter $m$ may be chosen arbitrarily large. Hence the theorem gives counterexamples for arbitrarily large $n$. Moreover, for these examples we have $n+t=2t+2m+1$, which is always odd.  Consequently, Theorem \ref{thm1} disproves Conjecture \ref{conj:oddcase} of Balogh, Linz, and Patk\'os.
	
	The counterexamples arise from the observation that the candidate family $\mathcal{B}_T(t,k)$ implicitly assumes that, on the critical layer $q=(n+t-1)/2$, the optimal $t$-intersecting subfamily is the fixed $t$-star $\mathcal {S}_T$. However, the relevant quantity is not merely the size of this layer, but rather the contribution $|\mathcal G|-|\nabla_{q+k}(\mathcal G)|$. Thus, one should compare $t$-intersecting $q$-uniform families according to the gain they provide on the $q$-layer minus the cost of the $(q+k)$-shade that must be deleted in order to destroy long chains. 
	
	\section{Proof of the Main Theorem}
	
	\begin{proof}[Proof of Theorem \ref{thm1}]
		We construct the desired family explicitly.  
		Partition the set into two disjoint parts
		$$[n]=X\cup Y,\qquad |X|=2m-1,\qquad |Y|=t+2.$$
		Define a $q$-uniform family
		\begin{equation*}
			\mathcal{G}=\left\{G\in\binom{[n]}{q}: |G\cap Y|\ge t+1\right\}.
		\end{equation*}
		Now, we let
		\begin{equation}\label{eq3}
			\mathcal{F}=\mathcal{G}
			\cup\bigcup_{i=1}^{k-1}\binom{[n]}{q+i}
			\cup\left(\binom{[n]}{q+k}\setminus \nabla_{q+k}(\mathcal{G})\right).
		\end{equation}
		We first verify that this family has the required structural properties.
		
		\begin{claim}
			The family $\mathcal{F}$ is $t$-intersecting.
		\end{claim}
		
		\begin{proof}
			Let $G_1,G_2\in\mathcal{G}$.  Since each of $G_1$ and $G_2$ contains at least $t+1$ elements of $Y$, while $|Y|=t+2$, we have
			$$|G_1\cap G_2\cap Y|\ge 2(t+1)-(t+2)=t.$$
			Thus $|G_1\cap G_2|\ge t$.
			
			Next let $G\in\mathcal{G}$ and let $A\subseteq[n]$ satisfy $|A|\ge q+1$.  Since $q=(n+t-1)/2=t+m$ and $n=t+2m+1$, we obtain
			$$|G\cap A|\ge |G|+|A|-n\ge q+(q+1)-n=t.$$
			
			Finally, if $A,B\subseteq[n]$ both have size at least $q+1$, then
			$$|A\cap B|\ge 2(q+1)-n=t+1\ge t.$$
			
			Every member of $\mathcal{F}$ either lies in $\mathcal{G}$ or has size at least $q+1$.  Hence any two members of $\mathcal{F}$ intersect in at least $t$ elements.
		\end{proof}
		
		\begin{claim}
			The family $\mathcal{F}$ is $k$-Sperner.
		\end{claim}
		
		\begin{proof}
			The family $\mathcal{F}$ is supported on the $k+1$ layers $q,q+1,\dots,q+k$.
			Suppose, for a contradiction, that $\mathcal{F}$ contains a chain $A_0\subset A_1\subset\cdots\subset A_k$.
			Each inclusion is strict, so the sizes must increase by at least one at each step.  
			Since the available layers range from $q$ to $q+k$, 
			the only possible sequence of sizes is $|A_i|=q+i$ for all $0\le i\le k$.
			In particular, $A_0\in\mathcal{G}$, $|A_k|=q+k$, and $A_0\subset A_k$.  This is exactly the condition that
			$$A_k\in\nabla_{q+k}(\mathcal{G}),$$
			but all such top-layer sets were removed in the definition of $\mathcal{F}$ in \eqref{eq3}.  This contradiction shows that no chain of length $k+1$ is contained in $\mathcal{F}$.
		\end{proof}
		
		It remains to compare the size of $\mathcal{F}$ with that of the fixed-star construction $\mathcal{B}_T(t,k)$.  The full intermediate layers in \eqref{eq1} and \eqref{eq3} are identical. Therefore the difference comes only from the bottom layer and from the deleted part of the top layer:
		\begin{equation}\label{eq4}
			|\mathcal{F}|-|\mathcal{B}_T(t,k)|
			=\left(|\mathcal{G}|-|\nabla_{q+k}(\mathcal{G})|\right)
			-\left(|\mathcal{S}_T|-|\nabla_{q+k}(\mathcal{S}_T)|\right).
		\end{equation}
		
		We now compute the two bracketed terms. For the fixed-star construction, using $n-t=2m+1$, $q-t=m$, and $q+k-t=m+k$, we have
		$$|\mathcal{S}_T|=\binom{n-t}{q-t}=\binom{2m+1}{m}$$
		and
		$$|\nabla_{q+k}(\mathcal{S}_T)|=\binom{n-t}{q+k-t}=\binom{2m+1}{m+k}.$$
		Hence
		\begin{equation}\label{eq5}
			|\mathcal{S}_T|-|\nabla_{q+k}(\mathcal{S}_T)|
			=\binom{2m+1}{m}-\binom{2m+1}{m+k}.
		\end{equation}
		
		For $\mathcal{G}$, a set $G\in\mathcal{G}$ contains either $t+1$ or $t+2$ elements of $Y$.  Since $m\ge k\ge 2$, this gives
		\begin{equation}\label{eq6}
			|\mathcal{G}|=(t+2)\binom{2m-1}{m-1}+\binom{2m-1}{m-2}.
		\end{equation}
		
		We next determine the top-layer shadow of $\mathcal{G}$.  A set $A\in\binom{[n]}{q+k}$ contains a member of $\mathcal{G}$ if and only if $|A\cap Y|\ge t+1$.  The forward implication is immediate from the definition of $\mathcal{G}$.  Conversely, suppose $|A\cap Y|\ge t+1$.  If $|A\cap Y|=t+1$, then
		$$|A\cap X|=q+k-(t+1)=m+k-1\ge m-1,$$
		so $A$ contains a $q$-set with exactly $t+1$ elements in $Y$.  If $|A\cap Y|=t+2$, then
		$$|A\cap X|=q+k-(t+2)=m+k-2\ge m-2,$$
		so $A$ contains a $q$-set with all $t+2$ elements of $Y$.  In either case, $A$ contains a member of $\mathcal{G}$.  Therefore
		\begin{equation}\label{eq7}
			|\nabla_{q+k}(\mathcal{G})|=(t+2)\binom{2m-1}{m+k-1}+\binom{2m-1}{m+k-2}.
		\end{equation}
		
		For convenience, we write
		$$C_m=\binom{2m-1}{m-1},\qquad R_{m,k}=\frac{(m!)^2}{(m-k)!(m+k)!}.$$
		
		For completeness, we record the binomial-ratio identities used in the simplification. After dividing by $C_m$, we deduce that
		\[\begin{aligned}
			\frac{\binom{2m-1}{m-2}}{C_m}&=\frac{m-1}{m+1},\\
			\frac{\binom{2m-1}{m+k-1}}{C_m}&=\frac{(m-1)! m!}{(m-k)!(m+k-1)!}=R_{m,k}\frac{m+k}{m},\\
			\frac{\binom{2m-1}{m+k-2}}{C_m}&=\frac{(m-1)! m!}{(m-k+1)!(m+k-2)!}=R_{m,k}\frac{(m+k)(m+k-1)}{m(m-k+1)},\\
			\frac{\binom{2m+1}{m}}{C_m}&=\frac{2(2m+1)}{m+1},\\
			\frac{\binom{2m+1}{m+k}}{C_m}&=\frac{2(2m+1)(m!)^2}{(m-k+1)!(m+k)!}=R_{m,k}\frac{2(2m+1)}{m-k+1}.
		\end{aligned}\]
		It follows from \eqref{eq4}, \eqref{eq5}, \eqref{eq6}, and \eqref{eq7} that
		\[\begin{aligned}
			\frac{|\mathcal{F}|-|\mathcal{B}_T(t,k)|}{C_m}
			={}&(t+2)+\frac{m-1}{m+1}
			-(t+2)R_{m,k}\frac{m+k}{m}  \\
			&-R_{m,k}\frac{(m+k)(m+k-1)}{m(m-k+1)}
			-\frac{2(2m+1)}{m+1}
			+R_{m,k}\frac{2(2m+1)}{m-k+1}.
		\end{aligned}\]
		The terms not involving $R_{m,k}$ simplify as
		$$(t+2)+\frac{m-1}{m+1}-\frac{2(2m+1)}{m+1}=t-1.$$
		For the remaining coefficient of $R_{m,k}$, after taking the common denominator
		$m(m-k+1)$, we obtain
		\[\begin{aligned}
			&-(t+2)\frac{m+k}{m}-\frac{(m+k)(m+k-1)}{m(m-k+1)}+\frac{2(2m+1)}{m-k+1} \\
			&=\frac{-(t+2)(m+k)(m-k+1)-(m+k)(m+k-1)+2m(2m+1)}{m(m-k+1)}  \\
			&=-\frac{\bigl((t-1)m+(t+1)k\bigr)(m-k+1)}{m(m-k+1)} \\
			&=-(t-1)-\frac{k(t+1)}{m}.
		\end{aligned}\]
		Consequently, we get
		\begin{align}\label{eq8}
			\frac{|\mathcal{F}|-|\mathcal{B}_T(t,k)|}{C_m}
			&=t-1+R_{m,k}\left(-(t-1)-\frac{k(t+1)}{m}\right) \nonumber\\
			&=(t-1)(1-R_{m,k})-R_{m,k}\frac{k(t+1)}{m} \nonumber\\
			&=R_{m,k}\left((t-1)(R_{m,k}^{-1}-1)-\frac{k(t+1)}{m}\right).
		\end{align}

		It remains to show that the right-hand side of \eqref{eq8} is positive.  Since $m\ge k$, all factors below are well-defined and positive.  
		We have
		$$R_{m,k}^{-1}
		=\frac{(m-k)!(m+k)!}{(m!)^2}
		=\frac{(m+k)\cdots(m+1)}{(m-k+1)\cdots (m)}
		=\prod_{i=0}^{k-1}\frac{m+i+1}{m-i}
		=\prod_{i=0}^{k-1}\left(1+\frac{2i+1}{m-i}\right).$$
		Because $k\ge 2$, expanding the product yields the strict inequality
		$$R_{m,k}^{-1}-1
		>\sum_{i=0}^{k-1}\frac{2i+1}{m-i}
		\ge\sum_{i=0}^{k-1}\frac{2i+1}{m}
		=\frac{k^2}{m}.$$
		Using the hypothesis $k(t-1)\ge t+1$, we obtain
		$$\begin{aligned}
			(t-1)(R_{m,k}^{-1}-1)-\frac{k(t+1)}{m}
			>\frac{(t-1)k^2-k(t+1)}{m}  =\frac{k\bigl(k(t-1)-(t+1)\bigr)}{m}
			\ge 0.
		\end{aligned}$$
		Thus the bracket in \eqref{eq8} is strictly positive.  
		Since $C_m>0$ and $R_{m,k}>0$, it follows that
		$$|\mathcal{F}|-|\mathcal{B}_T(t,k)|>0.$$
		This completes the proof of Theorem \ref{thm1}.
	\end{proof}

	\section*{Acknowledgement}
	\noindent This research is supported by National Key R\&D Program of China under grant number 2024YFA1013900, NSFC under grant number 12471327,  JSPS KAKENHI Grant Number 25KF0036, the NSF of Hubei Province Grant Number 2025AFB309,  the China Postdoctoral Science Foundation Grant Number 2025M773113, the Fundamental Research Funds for the Central Universities, Central China Normal University Grant Number CCNU24XJ026.
	
	\section*{Declaration}
	
	\noindent$\textbf{Conflict~of~interest}$
	The authors declare that they have no known competing financial interests or personal relationships that could have appeared to influence the work reported in this paper.
	
	\noindent$\textbf{Data~availability}$
	No data was used for the research described in the article.

\end{document}